\documentclass[11pt,letterpaper,reqno]{amsart}
\usepackage{hyperref}
\usepackage{amssymb}
\usepackage{amsmath,amscd,color,amsfonts}
\DeclareSymbolFontAlphabet{\mathbb}{AMSb}
\DeclareMathAlphabet{\mathbscr} {U}{BOONDOX-cal}{r}{n}
\usepackage{enumitem}
\usepackage{combelow}
\usepackage[all]{xy}
\usepackage{enumerate}
\usepackage{mathrsfs}
\usepackage{tensor}
\usepackage{stackrel}
\usepackage{marginnote}
\usepackage[normalem]{ulem}
\usepackage{comment}
\usepackage{tikz}

\providecommand{\U}[1]{\protect\rule{.1in}{.1in}}
\newtheorem{theorem}{Theorem}[section]
\newtheorem*{theorem*}{Theorem}
\newtheorem*{corollary*}{Corollary}
\newtheorem*{claim*}{Claim}

\newtheorem{corollary}[theorem]{Corollary}

\newtheorem{definition}[theorem]{Definition}

\newtheorem{example}[theorem]{Example}
\newtheorem{lemma}[theorem]{Lemma}
\newtheorem{proposition}[theorem]{Proposition}
\newtheorem*{proposition*}{Proposition}
\newtheorem{remark}[theorem]{Remark}
\numberwithin{equation}{section}

\newcommand{\im}{\operatorname{Im}} 
\renewcommand{\d}{\mathrm{d}} 
\newcommand{\R}{\mathbb{R}}         
\newcommand{\C}{\mathbb{C}}         
\newcommand{\Tt}{\mathbb{T}}         

\newcommand{\F}{\mathcal{F}}            
\renewcommand{\O}{\mathcal{O}}          

\newcommand{\sslash}{\mathbin{/\mkern-6mu/}}

\newcommand{\twprod}{\mathbin{%
    \ooalign{\raise1.15ex\hbox{$\scriptstyle\sim$}\cr\hidewidth$\times$\hidewidth\cr}%
    }}

\begin{document}
\title{On the chamber decomposition of a Hamiltonian torus action}

\author{Maarten Mol}
\address{Department of Mathematics, University of Toronto, 40 St. George Street, Toronto, ON M5S 2E4 Canada}
\email{maarten.mol.math@gmail.com}

\begin{abstract} 
For a compact connected Hamiltonian
$T$-space (with $T$ a torus), the components of the set of (relative) regular values of
the momentum map are open convex polytopes, whose closures (called the chambers) form a polyhedral decomposition of the momentum polytope. In a previous paper we showed that there is a canonical affine stratification of the momentum polytope with the property that, while varying through a stratum, the fibers of the momentum map do not change as $T$-spaces. In this paper we show that
the strata of this stratification are the relative interiors of the faces of
the chambers.  In doing so, we also give a proof of the fact that these faces form a polyhedral complex, which (it seems) was missing from the literature. We further extend the above to Hamiltonian $T$-spaces whose momentum map is proper as map into a convex set and, in the compact K\"{a}hler case, we point out a description of the above strata in terms of the $T_\mathbb{C}$-orbit closure polytopes.
\end{abstract}
\maketitle

\setcounter{tocdepth}{1}
\tableofcontents

\section{Introduction}
It is a classical theorem due to Atiyah \cite{At82}, as well as Guillemin and Sternberg \cite{GS82}, that the momentum map image of a Hamiltonian torus action on a compact connected symplectic manifold is a convex polytope. A subsequent theorem, due to Delzant \cite{De88}, shows that the equivariant symplectic geometry of a symplectic manifold with a Hamiltonian torus action is entirely encoded by the corresponding polytope when all of its symplectic reduced spaces are points (and the action is effective). For example, in that case the face structure of the polytope corresponds to the orbit type stratification of the torus action and the set of regular values attained by the momentum map is the interior of the polytope. The geometry of Hamiltonian torus actions whose symplectic reduced spaces are not just points is much more intricate. For instance, for such torus actions the set of regular values attained by the momentum map is not necessarily equal to the interior of the polytope. Instead, the following holds.
\begin{theorem}[Folklore]\label{thm:folklore:intro} Let $\mu:(S,\omega)\to \mathfrak{t}^*$ be a compact and connected Hamiltonian $T$-space. The connected components of the set of relative regular values attained by $\mu$ are open convex subsets of the momentum polytope $\Delta:=\mu(S)$ and their closures are convex polytopes, the faces of which form a polyhedral complex.  
\end{theorem}
Here, by relative regular values of $\mu$ we mean regular values of $\mu$ when considered as map into the affine hull of its image $\Delta$. The convex polytopes in this theorem are usually called the chambers of $\Delta$ and their codimension-one faces are often called walls (see, e.g., \cite{BuTer26,ChSaSe26,GoHoJe,GLS96,GS89,Met00,Par11,WaYa26}). \\

Whilst the convexity part of Theorem \ref{thm:folklore:intro} seems to have been stated for the first time in \cite{GS89} (and  analogues of Theorem \ref{thm:folklore:intro} have been proved in geometric invariant theory -- cf. \cite{BriPro90,DolHu98,Res00,Thad96}), until very recently \cite{Liu26} a proof of this convexity statement had not appeared in the literature. In this paper we give a proof of the full statement of Theorem \ref{thm:folklore:intro} and explain the relationship of the decomposition into chambers with the stratification in the following theorem.  
\begin{theorem}[\cite{Mol}]\label{thm:momimstrat} Let $\mu:(S,\omega)\to \mathfrak{t}^*$ be a Hamiltonian $T$-space with $\mu$ proper as map into its image $\Delta:=\mu(S)$. There is a natural affine stratification $\mathcal{S}_\mathrm{Ham}(\Delta)$ of $\Delta$ (see Definition \ref{def:affstrat}) with the property that, for each stratum $\sigma$, the restriction
\[
\mu:\mu^{-1}(\sigma)\to \sigma
\] is an equivariantly locally trivial fibration.
\end{theorem}
 The local triviality in this statement is meant in a sense that is stronger than topological local triviality and implies that the fibers and the reduced spaces (which are the quotients of the fibers) at any two values in the same stratum are not just homeomorphic, but isomorphic as differentiable stratified spaces with the orbit type stratification from \cite{LeSja91}. \\
 
 In \cite{Mol} we in fact established a version of Theorem \ref{thm:momimstrat} for Hamiltonian actions of arbitrary compact Lie groups (and, more generally, of proper quasi-symplectic groupoids). For Hamiltonian torus actions the construction of the stratification $\mathcal{S}_\mathrm{Ham}(\Delta)$ in \cite{Mol} boils down locally to taking appropriate intersections of the images of the $T$-orbit type strata in $S$ under the momentum map (as is recalled in more precise terms in Section \ref{sec:bckgr:hamstrat}). The result of this construction is related to the chamber decomposition as follows. 
 
\begin{theorem}\label{thm:main:stratcomparison} In the setting of Theorem \ref{thm:folklore:intro},  $\mathcal{S}_\mathrm{Ham}(\Delta)$ coincides with the stratification of $\Delta$ by relative interiors of the faces of its chambers. 
\end{theorem}
Proving this theorem was, in fact, our initial motivation for this paper. It shows, in particular, that the momentum map is equivariantly locally trivial along the relative interiors of the faces of the chambers and, consequently, that the reduced spaces of $\mu$ at any two values in such a relative interior are isomorphic as differentiable stratified spaces. On the other hand, it shows that the strata of $\mathcal{S}_\mathrm{Ham}(\Delta)$ are convex sets with convex polytopes as closures. \\

For Hamiltonian $T$-actions on compact K\"{a}hler manifolds (with $T$ assumed to act by biholomorphisms), Atiyah \cite{At82} also showed that the images of orbit closures of the induced holomorphic action of the complexified torus $T_\mathbb{C}$ are convex polytopes. This gives a collection of polytopes (with possibly overlapping interiors) that cover the momentum map image $\Delta$. In \cite{GorMac87} Goresky and MacPherson explained (in the algebraic context of projective varieties) how one can reconstruct the $T$-orbit space (as topological space) from this collection of polytopes and certain additional invariants involving the symplectic reduced spaces. To do so, they considered a natural partition of $\Delta$ associated to this collection of polytopes (which is, in fact, a stratification with possibly disconnected strata), with the property that the set of semistable points for the shifted momentum map $\mu-x$ is the same for all points $x$ in the same member of this partition. In particular, like for the stratification $\mathcal{S}_\mathrm{Ham}(\Delta)$ in Theorem \ref{thm:momimstrat}, the symplectic reduced spaces at any two such points $x$ are homeomorphic. In addition to proving Theorem \ref{thm:folklore:intro} and Theorem \ref{thm:main:stratcomparison}, we show that in this setting $\mathcal{S}_\mathrm{Ham}(\Delta)$ in fact equals the collection of connected components of this partition, thus giving another characterization of $\mathcal{S}_\mathrm{Ham}(\Delta)$ in the K\"{a}hler case. Moreover, we point out that if the K\"{a}hler manifold is connected, then the members of this partition are in fact certain intersections of relative interiors of the $T_\mathbb{C}$-orbit closure polytopes. This leads to the following conclusion for $\mathcal{S}_\mathrm{Ham}(\Delta)$.
\begin{proposition}\label{prop:orbitclosurepolytopesintersection} If the K\"{a}hler manifold above is connected, then for any $x\in \Delta$ the stratum of $\mathcal{S}_\mathrm{Ham}(\Delta)$ through $x$ equals the intersection of all the relative interiors $\mathring{P}$ of $T_\mathbb{C}$-orbit closure polytopes $P$ for which $x\in \mathring{P}$.
\end{proposition}

\textbf{\underline{Outline of the paper:}} 
In the main body of the paper we, in fact, state and prove Theorem \ref{thm:momim:locpoldec} and Theorem \ref{thm:momim:stratcomparison}, which are somewhat more general versions of Theorem \ref{thm:folklore:intro} and Theorem \ref{thm:main:stratcomparison} that also apply to Hamiltonian torus actions on non-compact connected symplectic manifolds, as long as the momentum map is proper as map into some convex set. To prove these, in Section \ref{sec:affstrat:chamberdec} we first explain a general bijective correspondence between  locally polyhedral chamber decompositions of certain (such as closed or convex) subsets of real finite-dimensional vector spaces on one hand, and certain affine stratifications of such subsets on the other, in which the faces of chambers are exactly the closures of the strata of the corresponding stratification (see Theorem \ref{thm:affstrat:chamberdec:correspondence}). What characterizes the stratifications in this correspondence is that they (and the subsets that they stratify) are locally modeled on fans (in the sense of Definition \ref{def:affstrat:locfan}). It therefore gives a characterization of locally polyhedral chamber decompositions that is more local in nature and reduces the proofs of Theorem \ref{thm:momim:locpoldec} and Theorem \ref{thm:momim:stratcomparison} to showing that the momentum map image $\Delta$ with the stratification $\mathcal{S}_\mathrm{Ham}(\Delta)$ is locally modelled on fans. This is done in Section \ref{sec:chamberdecmomim}, by using the Marle-Guillemin-Sternberg normal form theorem to show that, near a given point $x\in \Delta$, the stratification $\mathcal{S}_\mathrm{Ham}(\Delta)$ is modelled on a fan $\mathcal{F}_x$ constructed from the weights of the symplectic normal representations at points in $p\in \mu^{-1}(x)$ by combining the Gelfand-Kapranov-Zelevinsky fans of the vector configurations formed by these weights. At the end of the paper, in Section \ref{sec:orbitclosures}, we specialize to the K\"{a}hler case, compare the stratification in \cite{GorMac87} to $\mathcal{S}_\mathrm{Ham}(\Delta)$ and  address Proposition \ref{prop:orbitclosurepolytopesintersection}. \\

\textbf{\underline{Acknowledgements:}} I am grateful to Eckhard Meinrenken and Daniele Sepe for their encouragement to write this paper, and I am particularly indebted to Eckhard for suggesting the key idea behind the proof of Theorem \ref{thm:convexity:strata:general}. AI (a free version) was used to experiment with Proposition \ref{prop:partitionmember:intersection}, but did not contribute to proofs, other mathematical content, or writing in this paper. This research was supported by NSERC grant RGPIN-2024-05764. 

\section{Affine stratifications and chamber decompositions}\label{sec:affstrat:chamberdec} 
Let $V$ be a finite-dimensional real vector space. In this section we formalize what we mean by a locally polyhedral chamber decomposition of a subset $X\subset V$ and, for those $X$ that are closed in a convex subset of $V$ (such as $X\subset V$ that are closed or convex in $V$), we characterize such decompositions in terms of stratifications by explaining and proving the following. 
\begin{theorem}\label{thm:affstrat:chamberdec:correspondence}
    For any $X
    \subset V$ that is closed in a convex subset of $V$, there is a canonical bijection
\[ 
\left\{\\ \txt{Stratifications $\mathcal{S}$ of $X$ for which \\ $(X,\mathcal{S})$ is locally modelled on fans \,}\right\}\ 
\tilde{\longleftrightarrow}\ 
\left\{\\ \txt{Locally polyhedral \\ chamber decompositions of $X$\,}\right\}
\] 
that associates to such a stratification $\mathcal{S}$ the decomposition consisting of the closures in $X$ of the strata of $\mathcal{S}$ and to such a decomposition of $X$ the stratification in Proposition \ref{prop:locpoldec} below. 
\end{theorem} 
The proof that we give of this will in particular show the following. 
\begin{theorem}\label{thm:convexity:strata:general}
    For any stratification $\mathcal{S}$ as in the correspondence in Theorem \ref{thm:affstrat:chamberdec:correspondence}, the strata are convex. 
\end{theorem}
\subsection{Stratifications of locally polyhedral chamber decompositions}\label{sec:stratoflocpoldec}
In this subsection we explain both sides of the correspondence in Theorem \ref{thm:affstrat:chamberdec:correspondence}, as well as the map from right to left. For right-hand side, first recall that a \textbf{polyhedral complex} in $V$ is a collection $\mathcal{P}$ of (non-empty) polyhedral sets in $V$ such that 
\begin{itemize}
\item[(i)] the faces of any $P\in \mathcal{P}$ also belong to $\mathcal{P}$,
\item[(ii)] for any $P,Q\in \mathcal{P}$ that intersect, $P\cap Q$ is a face of both $P$ and $Q$.
\end{itemize}
The \textbf{support} of a polyhedral complex $\mathcal{P}$ is the subset 
\[
\textrm{supp}(\mathcal{P}):=\bigcup_{P\in \mathcal{P}} P\subset V
\]
given by the union of all members of $\mathcal{P}$. By a \textbf{polyhedral chamber decomposition of $X$} we will mean a locally finite polyhedral complex with support equal to $X$. A locally polyhedral chamber decomposition will be defined as a similar collection that consists of sets that are locally polyhedral (rather than polyhedral), in the following sense. 
\begin{definition}\label{def:locpolset} A subset $P$ of $V$ is called \textbf{locally polyhedral} if at every $x\in P$ the germ of $P$ at $x$ is equal to the germ of $x+\sigma_x$ at $x$, for some polyhedral cone $\sigma_x$ in $V$. That is, there is an open neighbourhood $U$ of $x$ in $V$ such that $P\cap U=(x+\sigma_x)\cap U$. 
\end{definition} Given such $P$, for each $x\in P$ the cone $\sigma_x$ in Definition \ref{def:locpolset} is uniquely determined. We will denote it as $\sigma_x(P)$. Consider the equivalence relation
\begin{equation}\label{eqn:eqrel:strat:locpolset}
x\sim y \iff F_x(P)=F_y(P),
\end{equation}
where $F_x(P)$ and $F_y(P)$ are the  minimal faces of $\sigma_x(P)$ and $\sigma_y(P)$. 
\begin{proposition}\label{prop:locpolstrat} Any locally polyhedral subset $P\subset V$ comes with a stratification $\mathcal{S}(P)$ consisting of the connected components of the members of the partition of $P$ given by the equivalence relation \eqref{eqn:eqrel:strat:locpolset}. 
\end{proposition}
 As this is surely known, we leave its proof to the reader. 
 \begin{remark}\label{rem:stratification:polyhedron} For a polyhedron $P$ the stratification $\mathcal{S}(P)$ consists of the relative interiors of the (non-empty) faces of $P$. 
 \end{remark} 
 Using the stratification in Proposition \ref{prop:locpolstrat}, we now give a precise definition of a locally polyhedral chamber decomposition. 
\begin{definition}\label{def:locallypolyhedralchamberdecomposition}
    Let $X\subset V$. By a \textbf{locally polyhedral chamber decomposition} of $X$ we mean a locally finite cover $\mathcal{P}$ of $X$ consisting of (non-empty) locally polyhedral subsets of $V$ satisfying:
    \begin{itemize}
\item[(i)] for each $P\in \mathcal{P}$ and $\Sigma\in \mathcal{S}(P)$, the closure of $\Sigma$ in $X$ belongs to $\mathcal{P}$,

\item[(ii)] for any $P,Q\in \mathcal{P}$ that intersect, the intersection $P\cap Q$ is the closure in $X$ of a common stratum $\Sigma\in \mathcal{S}(P)\cap \mathcal{S}(Q)$.
    \end{itemize}
\end{definition} 
In view of Remark \ref{rem:stratification:polyhedron}, this indeed extends the definition of polyhedral chamber decomposition that we gave before. 
\begin{remark}\ 
\begin{itemize}\item[(i)] Because of condition (ii) in Definition \ref{def:locallypolyhedralchamberdecomposition} for $P=Q$, any member of a locally polyhedral chamber decomposition must be connected. 
\item[(ii)] If $X$ is compact, then any locally polyhedral chamber decomposition of $X$ is, in fact, polyhedral and consists of convex polytopes, because any compact and connected locally polyhedral set is a convex polytope.
\end{itemize}
\end{remark}
\begin{remark}\label{rem:maxelts:locpolchambdec} As for polyhedral complexes, any locally polyhedral chamber decomposition $\mathcal{P}$ is determined by its set of maximal elements with respect to inclusion, since $\mathcal{P}$ equals the collection of closures in $X$ of strata of its maximal elements.
\end{remark}
Next, we turn to the left-hand side of the correspondence in Theorem \ref{thm:affstrat:chamberdec:correspondence}. To begin with, recall that a \textbf{fan} is a finite polyhedral complex consisting of polyhedral cones, and that the support of a fan is partitioned by its \textbf{relatively open cones} (meaning the relative interiors of its cones). 

\begin{definition}\label{def:affstrat:locfan} We call a stratified subset $(X,\mathcal{S})$ of $V$ \textbf{locally modelled on fans} if for each $x_0\in X$ there is a fan $\mathcal{F}_{x_0}$ in $V$ and an open neighbourhood $U$ of $x_0$ in $V$ such that the partition of $X\cap U$ by connected components of the subspaces $\Sigma\cap U$, with $\Sigma\in \mathcal{S}$, is equal to the collection
\begin{equation}\label{eqn:localstratafan}
\{(x_0+\mathring{\sigma})\cap U\mid \mathring{\sigma} \text{ a relatively open cone of }\F_{x_0}
\}.
\end{equation} 
\end{definition}
 
\begin{remark}\
\begin{itemize} 
 \item[(i)] For any stratified subset $(X,\mathcal{S})$ of $V$ that is locally modelled on fans and any $x_0\in X$, the fan $\mathcal{F}_{x_0}$ is uniquely determined. Moreover, the germ of $X$ in $V$ at $x_0$ is the germ of $x_0+\mathrm{Supp}(\F_{x_0})$.  
 \item[(ii)] If $(X,\mathcal{S})$ is a stratified subspace of $V$ that is locally modelled on fans, then the opens $U$ in Definition \ref{def:affstrat:locfan} can always be taken to be convex by passing to any smaller convex open neighbourhood.
 \end{itemize}
 \end{remark}
 \begin{example} The stratified subset depicted below, in which $X$ is a square, is not locally modeled on a fan at the midpoint of the left edge of the square. 
 
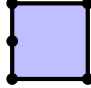
\begin{center}
\begin{tikzpicture} 
\filldraw[blue!25!white] (-1,0) -- (-1,1) -- (0,1) -- (0,0) -- cycle ;

\draw[ultra thick,black] (-1,0) -- (0,0) ;
\draw[ultra thick,black] (-1,0) -- (-1,1) ;
\draw[ultra thick,black] (0,0) -- (0,1) ;
\draw[ultra thick,black] (-1,1) -- (0,1) ;

\filldraw[black] (0,0) circle (2pt) ;
\filldraw[black] (-1,0) circle (2pt) ;
\filldraw[black] (0,1) circle (2pt) ;
\filldraw[black] (-1,1) circle (2pt) ;
\filldraw[black] (-1,0.5) circle (2pt) ;

\end{tikzpicture}
\end{center}
\end{example}
 \begin{example} The stratified subset depicted below, in which $X$ is a rectangle with a closed line segment deleted, is locally modeled on fans.   
 
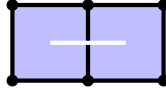
\begin{center}
\begin{tikzpicture} 
\filldraw[blue!25!white] (-1,0) -- (-1,1) -- (1,1) -- (1,0) -- cycle ;

\draw[ultra thick,black] (-1,0) -- (1,0) ;
\draw[ultra thick,black] (-1,0) -- (-1,1) ;
\draw[ultra thick,black] (-1,1) -- (1,1) ;
\draw[ultra thick,black] (0,1) -- (0,0.5) ;
\draw[ultra thick,black] (0,0) -- (0,0.5) ;
\draw[ultra thick,black] (1,0) -- (1,1) ;
\draw[ultra thick,white] (-0.5,0.5)--(0.5,0.5);

\filldraw[black] (0,0) circle (2pt) ;
\filldraw[black] (-1,0) circle (2pt) ;
\filldraw[black] (1,0) circle (2pt) ;
\filldraw[black] (-1,1) circle (2pt) ;
\filldraw[black] (0,1) circle (2pt) ;
\filldraw[black] (1,1) circle (2pt) ;
\end{tikzpicture}
\end{center} 
This shows that, when $X$ is not closed in a convex subset of $V$, $(X,\mathcal{S})$ being locally modeled on fans does not guarantee that the closures in $X$ of its strata form a locally polyhedral chamber decomposition.
 \end{example}

 Stratifications as in Definition \ref{def:affstrat:locfan} are affine, in the following sense. 
 \begin{definition}\label{def:affstrat}
   By an \textbf{affine stratification} of a subset $X\subset V$, we mean a stratification such that each stratum is a relatively open subset of $V$ (meaning that it is open in its affine hull).  
\end{definition} 
The following provides a criterion for subsets of $V$ with affine stratifications to be locally modeled on fans, which will be useful later on. 
 \begin{proposition}\label{prop:affstrat:locmodfan:tangentcrit}
     Suppose that $X\subset V$ and $\mathcal{S}$ is an affine stratification of $X$. Then for any $x_0\in X$, any convex open neighbourhood $U$ of $x_0$ in $V$ and any fan $\mathcal{F}_{x_0}$, the condition that the collection of connected components of the sets $\Sigma\cap U$, with $\Sigma\in \mathcal{S}$, is equal to \eqref{eqn:localstratafan} is equivalent to requiring that $X\cap U=\big(x_0+\mathrm{supp}(\F_{x_0})\big)\cap U$ and that for each $x\in U$ the stratum $\Sigma\in \mathcal{S}$ through $x$ has tangent space
\[
T_x\Sigma=T_x(x_0+\mathring{\sigma}),
\] with $\mathring{\sigma}$ the relatively open cone of $\F_{x_0}$ through $x-x_0$. 
 \end{proposition}
 \begin{proof} This follows from \cite[Lemma 6]{CrMe18} and the fact that the germ at a point $x$ of a relatively open subset is determined by its tangent space at $x$.
 \end{proof}
 
In the rest of this section we prove the following, which in particular gives the map from right to left in Theorem \ref{thm:affstrat:chamberdec:correspondence} and proves its injectivity. 
\begin{proposition}\label{prop:locpoldec}
    If $X\subset V$ and $\mathcal{P}$ is a locally polyhedral chamber decomposition of $X$, then the collection \[
    \mathcal{S}(\mathcal{P})=\bigcup_{P\in \mathcal{P}}\mathcal{S}(P)
    \]
    of strata of the locally polyhedral sets in $\mathcal{P}$ is a stratification of $X$ and $(X,\mathcal{S}(\mathcal{P}))$ is locally modelled on fans. Moreover, $\mathcal{P}$ equals the collection of closures in $X$ of the strata of $\mathcal{S}(\mathcal{P})$.  
\end{proposition}
To prove this, we first note the following.
\begin{lemma}\label{lemma:locpoldec}
For $X$ and $\mathcal{P}$ as in Proposition \ref{prop:locpoldec}, $P\in \mathcal{P}$ and $\Sigma\in \mathcal{S}(P)$,
\begin{itemize}\item[(a)] the closure $\overline{\Sigma}$ of $\Sigma$ in $X$ equals the closure of $\Sigma$ in $P$, 
\item[(b)] $\Sigma$ is the unique open stratum of $\overline{\Sigma}$.
\end{itemize}
\end{lemma}
\begin{proof} By (ii) applied to the given $P$ and to $Q:=\overline{\Sigma}$ (which, by (i), belongs to $\mathcal{P}$), there is a common stratum of $P$ and $Q$ whose closure in $X$ is $P\cap \overline{\Sigma}$ (which is the closure of $\Sigma$ in $P$). Since $\Sigma$ is locally closed in $V$ (and, hence, in $X$), it is an open subset of $\overline{\Sigma}$. Therefore it must intersect this common stratum and, both being strata of $P$, $\Sigma$ must be  this common stratum. So, $\overline{\Sigma}$ indeed equals the closure of $\Sigma$ in $P$ and $\Sigma$ is  an open stratum of $\overline{\Sigma}$. Since $\Sigma$ is also dense in $\overline{\Sigma}$, it is the unique open stratum.
\end{proof}
Part (a) of this lemma, combined with the following, shows that the polyhedral cones $\sigma_x(\overline{\Sigma})$ of the closure $\overline{\Sigma}$ in $X$ of a stratum $\Sigma\in \mathcal{S}(P)$ are naturally related to those of $P$. 
\begin{lemma}\label{lemma:internalclosures:strata:locpolsets}
    Let $P$ be a locally polyhedral subset of $V$ and $\Sigma\in \mathcal{S}(P)$ a stratum. The closure $\overline{\Sigma}$ in $P$ is a locally polyhedral subset of $V$ as well and, for each $x\in \overline{\Sigma}$, the polyhedral cone $\sigma_x(\overline{\Sigma})$ is a face of $\sigma_x(P)$. 
\end{lemma}
Using these lemmas we will now complete the proof of the proposition. We leave the proof of Lemma \ref{lemma:internalclosures:strata:locpolsets} to the reader, for it is surely well-known.
\begin{proof}[Proof of Proposition \ref{prop:locpoldec}] To see that $\mathcal{S}(\mathcal{P})$ is a partition, we ought to show that if $\Sigma_1,\Sigma_2\in \mathcal{S}(\mathcal{P})$ intersect each other, then $\Sigma_1=\Sigma_2$. By (i) the closures $\overline{\Sigma}_1$ and $\overline{\Sigma}_2$ in $X$ belong to $\mathcal{P}$. So, by (ii), there is a $\Sigma\in \mathcal{S}(\overline{\Sigma}_1)\cap \mathcal{S}(\overline{\Sigma}_2)$ with closure $\overline{\Sigma}$ in $X$ equal to $\overline{\Sigma}_1\cap \overline{\Sigma}_2$. Since $\Sigma_1$ and $\Sigma_2$ are open in their closures, $\Sigma_1\cap \Sigma_2$ is open in $\overline{\Sigma}_1\cap \overline{\Sigma}_2$ and, so, must intersect $\Sigma$. As both $\Sigma_1$ and $\Sigma$ are strata of $\overline{\Sigma}_1$, and both $\Sigma_2$ and $\Sigma$ are strata of $\overline{\Sigma}_2$ (see part (b) of Lemma \ref{lemma:locpoldec}), it follows that $\Sigma_1=\Sigma=\Sigma_2$.

Next, we show that $\mathcal{S}(\mathcal{P})$ is in fact a stratification. Suppose  $\Sigma_1,\Sigma_2\in \mathcal{S}(\mathcal{P})$ are such that $\Sigma_1$ intersects the closure $\overline{\Sigma}_2$ in $X$. By (i), $\overline{\Sigma}_2$ belongs to $\mathcal{P}$. Since $\mathcal{S}(\mathcal{P})$ is a partition, $\Sigma_1$ must equal the stratum of $\overline{\Sigma}_2$ that it intersects. So, $\Sigma_1\subset \overline{\Sigma}_2$ and $\dim(\Sigma_1)<\dim(\Sigma_2)$, which proves that $\mathcal{S}(\mathcal{P})$ satisfies the frontier condition. Since $\mathcal{P}$ is locally finite, $\mathcal{S}(\mathcal{P})$ is so as well and, hence, it is indeed a stratification of $X$.

To prove that $\mathcal{S}(\mathcal{P})$ is locally modelled on fans, let $x_0\in X$. Since $\mathcal{S}(\mathcal{P})$ is locally finite and the closures of its strata are locally polyhedral, there is a convex open neighbourhood $U$ of $x_0$ in $V$ small enough so that $U$ only intersects strata $\Sigma$ for which $x_0\in \overline{\Sigma}$ and so that $\overline{\Sigma}\cap U=(x_0+\sigma_{x_0}(\overline{\Sigma}))\cap U$ for each such $\Sigma$. Then, because (by part (b) of Lemma \ref{lemma:locpoldec}) $\Sigma$ is the unique open stratum of $\overline{\Sigma}$, it also holds that $\Sigma\cap U=\mathring{\sigma}_{x_0}(\overline{\Sigma})\cap U$ for each such $\Sigma$, where $\mathring{\sigma}_{x_0}(\overline{\Sigma})$ denotes the relative interior of $\sigma_{x_0}(\overline{\Sigma})$. So, since $U$ intersects each such relative interior in a convex (and, hence, connected) set, once we show that the finite collection of polyhedral cones 
\[
\mathcal{F}_{x_0}:=\{\sigma_{x_0}(\overline{\Sigma})\mid \Sigma\in \mathcal{S}(\mathcal{\mathcal{P}})\text{ such that } x_0\in \overline{\Sigma}\}
\] 
is a fan, we can conclude that $\mathcal{S}(\mathcal{P})$ is locally modelled on fans. To this end, suppose that $\Sigma\in \mathcal{S}(\mathcal{P})$ such that $x_0\in \overline{\Sigma}$, and let $\tau$ be a face of $\sigma_{x_0}(\overline{\Sigma})$. Since $\overline{\Sigma}\cap U=(x_0+\sigma_{x_0}(\overline{\Sigma}))\cap U$, for any $x\in U$ the cone $\sigma_x(\overline{\Sigma})$ of $\overline{\Sigma}$ at $x$ is equal to that of $\sigma_{x_0}(\overline{\Sigma})$. In view of this and Remark \ref{rem:stratification:polyhedron} for $P=\sigma_{x_0}(\overline{\Sigma})$,  $\sigma_x(\overline{\Sigma})=\sigma_y(\overline{\Sigma})$ for all $x,y\in U$ that lie in the relative interior $\mathring{\tau}$ of $\tau$. Hence, $(x_0+\mathring{\tau})\cap U$ is contained in a stratum $\Sigma_\tau\in \mathcal{S}(\overline{\Sigma})$.  In particular, $x_0\in \overline{\Sigma}_\tau$ and so, by construction of $U$, it holds that $\Sigma_\tau\cap U= \mathring{\sigma}_{x_0}(\overline{\Sigma}_\tau)\cap U$. Therefore, $(x_0+\mathring{\tau})\cap U$ is contained in $\mathring{\sigma}_{x_0}(\overline{\Sigma}_\tau)$. So, $\mathring{\tau}$ intersects $\mathring{\sigma}_{x_0}(\overline{\Sigma}_\tau)$. Since $\tau$ is a face of $\sigma_{x_0}(\overline{\Sigma})$ and, by Lemma \ref{lemma:internalclosures:strata:locpolsets}, the cone $\sigma_{x_0}(\overline{\Sigma}_\tau)$ is so too, it follows from this that $\tau=\sigma_{x_0}(\overline{\Sigma}_\tau)$. Hence, $\tau\in \mathcal{F}_{x_0}$, which shows that $\mathcal{F}_{x_0}$ satisfies the first axiom of a fan. For the second, suppose that $\Sigma_1,\Sigma_2\in \mathcal{S}(\mathcal{P})$ such that $x_0\in \overline{\Sigma}_1\cap\overline{\Sigma}_2$. Then $\overline{\Sigma}_1\cap \overline{\Sigma}_2=\overline{\Sigma}$ for some $\Sigma\in \mathcal{S}$, by condition (ii) in Definition \ref{def:locallypolyhedralchamberdecomposition}. Moreover, 
\[
\overline{\Sigma}\cap U=\overline{\Sigma}_1\cap \overline{\Sigma}_2\cap U=(x_0+\sigma_{x_0}(\overline{\Sigma}_1)\cap \sigma_{x_0}(\overline{\Sigma}_2))\cap U.
\] Therefore $\sigma_{x_0}(\overline{\Sigma}_1)\cap \sigma_{x_0}(\overline{\Sigma}_2)=\sigma_{x_0}(\overline{\Sigma})$, which by Lemma \ref{lemma:internalclosures:strata:locpolsets} is a face of both $\sigma_{x_0}(\overline{\Sigma}_1)$ and $\sigma_{x_0}(\overline{\Sigma}_2)$. Thus, $\mathcal{F}_{x_0}$ is a fan, as was left to be shown. 

Finally, that $\mathcal{P}$ is contained in the collection of closures in $X$ of the strata of $\mathcal{S}(\mathcal{P})$ follows from condition (ii) in Definition \ref{def:locallypolyhedralchamberdecomposition} for $P=Q$, and the reverse inclusion follows from condition (i). 
\end{proof}
\begin{remark}\label{rem:maxelts=closuresopenstrata} Any stratification comes with a partial order given by
\[
\Sigma\leq \Sigma'\iff \Sigma\subset \overline{\Sigma'},
\]
and the maximal elements with respect to this are the open strata. Therefore, the maximal elements (as in Remark \ref{rem:maxelts:locpolchambdec}) of a locally polyhedral chamber decomposition $\mathcal{P}$ of $X$ are the closures in $X$ of the strata of $\mathcal{S}(\mathcal{P})$ that are open in $X$.
\end{remark}
\subsection{From stratifications to decompositions and convexity of strata}
Next, we complete the proof of Theorem \ref{thm:affstrat:chamberdec:correspondence} and give a proof of Theorem \ref{thm:convexity:strata:general}. We start with the latter, because it will be used it for the former. 
\begin{proof}[Proof of Theorem \ref{thm:convexity:strata:general}] Suppose that $X$ and $\mathcal{S}$ are as in Theorem \ref{thm:affstrat:chamberdec:correspondence}. Let $\Sigma\in \mathcal{S}$ be a stratum of $\mathcal{S}$ and let $y_0,y_1\in \Sigma$. Since $\Sigma$ is path-connected, there is a path $\gamma:[0,1]\to \Sigma$ from $y_0$ to $y_1$. Write $\ell_t$ for the compact  line segment between $y_0$ and $\gamma(t)$. To prove convexity of $\Sigma$, we ought to show that $\ell_1\subset \Sigma$. Argueing by contradiction, suppose that $\ell_1\not\subset \Sigma$. Then 
\begin{equation}\label{eqn:set1:convexityopenstrata}
\{t\in [0,1]\mid \ell_t\not\subset \Sigma\}
\end{equation} is non-empty and, so, has an infimum $t_0\in [0,1]$. Let $A$ be the affine hull of $\Sigma$. By convexity of $A$, the map 
\[
\Gamma:[0,1]^2\to V,\quad \Gamma(s,t)=(1-s)y_0+s\gamma(t),
\] takes values in $A$. Using continuity of $\Gamma$ and that $\Sigma$ is open in $A$ (because $\mathcal{S}$ is affine in the sense of Definition \ref{def:affstrat}), it follows from the tube lemma that \eqref{eqn:set1:convexityopenstrata} is closed in $[0,1]$ and, so, it contains its infimum. Therefore $\ell_{t_0}\not\subset \Sigma$ and, in particular, $t_0>0$. Since $\{s\in [0,1]\mid \Gamma(s,t_0)\not\in \Sigma\}$ is non-empty it has an infimum $s_0\in [0,1]$, and since it is closed in $[0,1]$ as well, it also contains its infimum. So, $\Gamma(s_0,t_0)\notin \Sigma$ and, in particular, $s_0\in ]0,1[$. On the other hand, $\Gamma(s_0,t_0)\in X$ since $X$ is closed in a convex subset of $V$. So, there are a convex open $U$ around $x_0:=\Gamma(s_0,t_0)$ and a fan $\F_{x_0}$ as in Definition \ref{def:affstrat:locfan}. Since $s_0\in ]0,1[$, $t_0\in ]0,1]$ and $\Gamma^{-1}(U)$ is open in $[0,1]^2$, there is a $\delta>0$ such that $
[s_0-\delta,s_0+\delta]\times [t_0-\delta,t_0]\subset \Gamma^{-1}(U)$.
Let $\sigma\in \F_{x_0}$ be the cone for which the connected component of $\Sigma\cap U$ containing $\Gamma([s_0-\delta,s_0[\times\{t_0\})$ equals $(x_0+\mathring{\sigma})\cap U$. This being a polyhedral cone, there are linear functions $\alpha_1,...,\alpha_n\in V^*$ and a linear subspace $V_\sigma\subset V$ such that  
\[
\sigma=V_\sigma\cap \bigcap_{i=1}^n \{\alpha_i\geq0\}\quad \text{and}\quad \mathring{\sigma}=V_\sigma \cap \bigcap_{i=1}^n \{\alpha_i>0\}. 
\] 
Note that $n>0$, since $x_0\notin \Sigma$. Because $\alpha_i(\Gamma(s,t_0)-x_0)$ is an affine function of $s\in [s_0-\delta,s_0+\delta]$ that is strictly positive when $s\in [s_0-\delta,s_0[$ and vanishes at $s_0$, it must be strictly negative when $s\in ]s_0,s_0+\delta]$. So, $\Gamma(s,t_0)\notin x_0+\sigma$ when $s\in ]s_0,s_0+\delta]$. On the other hand, $\Gamma(s,t)\in \Sigma\cap U$ when 
\[
(s,t)=\begin{cases} (s_0-\delta,t) &\text{ with }t\in [t_0-\delta,t_0],\\
(s,t_0-\delta) &\text{ with }s\in [s_0-\delta,s_0+\delta],\\
(s_0+\delta,t) &\text{ with }t\in [t_0-\delta,t_0[.
\end{cases} 
\] 
Therefore $\Gamma(s_0+\delta,t)$ belongs to the same connected component of $\Sigma\cap U$ as $\Gamma(s_0-\delta,t_0)$ for all $t\in [t_0-\delta,t_0[$. This component being equal to $(x_0+\mathring{\sigma})\cap U$, it follows that $\Gamma(s_0+\delta,t_0)\in x_0+\sigma$, which is a contradiction.   
\end{proof}
Using this, we show the following and conclude the proof of Theorem \ref{thm:affstrat:chamberdec:correspondence}. 
\begin{lemma}\label{lemma:strataclosures:locfanstrat} Let $X$ and $\mathcal{S}$ be as in Theorem \ref{thm:affstrat:chamberdec:correspondence}. Further, let $\Sigma\in \mathcal{S}$. 
\begin{itemize}\item[(a)] The closure $\overline{\Sigma}$ in $X$ of $\Sigma$ is a locally polyhedral subset of $V$.
\item[(b)] The cone $\sigma_{x}(\overline{\Sigma})$ belongs to $\mathcal{F}_{x}$ for any $x\in \overline{\Sigma}$.
\item[(c)] The stratification $\mathcal{S}(\overline{\Sigma})$ of the locally polyhedral set $\overline{\Sigma}$ (as in Proposition \ref{prop:locpolstrat}) equals the collection of strata of $\mathcal{S}$ that intersect $\overline{\Sigma}$.
\item[(d)] If $\Sigma_1,\Sigma_2\in \mathcal{S}$, the stratification $\mathcal{S}(\overline{\Sigma}_1\cap \overline{\Sigma}_2)$ of the locally polyhedral set $\overline{\Sigma}_1\cap \overline{\Sigma}_2$ equals the collection of strata of $\mathcal{S}$ that intersect $\overline{\Sigma}_1\cap \overline{\Sigma}_2$.
\end{itemize}
\end{lemma}
\begin{proof} Let $x\in \overline{\Sigma}$. Consider a convex open $U$ around $x$ and a fan $\mathcal{F}_{x}$ as in Definition \ref{def:affstrat:locfan}. By Theorem \ref{thm:convexity:strata:general}, $\Sigma$ is convex. So, $\Sigma\cap U$ is convex as well and, hence, it is connected. Therefore, $\Sigma\cap U=(x+\mathring{\sigma})\cap U$ for some  $\sigma\in \mathring{\F}_{x}$. Since $X\cap U=\big(x+\mathrm{supp}(\mathcal{F}_{x})\big)\cap U$ and $U$ is open in $V$, by taking closures in $X\cap U$ it follows that $\overline{\Sigma}\cap U=(x_0+\sigma)\cap U$. This shows that (a) and (b) hold. For (c), note that by the frontier condition on $\mathcal{S}$ the strata of $\mathcal{S}$ that intersect $\overline{\Sigma}$ form a stratification of $\overline{\Sigma}$, which is affine in the sense of Definition \ref{def:affstrat}, because $\mathcal{S}$ is. Because $(X,\mathcal{S})$ is locally modelled on fans, for each $x\in \overline{\Sigma}$ the tangent space to the stratum of $\mathcal{S}$ through $x$ is equal to the common minimal face of the cones in $\mathcal{F}_x$. It follows from (b) that the same holds for the strata of $\mathcal{S}(\overline{\Sigma})$. So, since $\mathcal{S}(\overline{\Sigma})$ is an affine stratification of $\overline{\Sigma}$ as well, by the same argument as for Proposition \ref{prop:affstrat:locmodfan:tangentcrit} these two stratifications of $\overline{\Sigma}$ must be equal. Part (d) follows in the same way, now using in addition that for any $x\in \overline{\Sigma}_1\cap \overline{\Sigma}_2$ the cone $\sigma_x(\overline{\Sigma}_1\cap\overline{\Sigma}_2)$ belongs to $\mathcal{F}_x$, because $\sigma_x(\overline{\Sigma}_1\cap\overline{\Sigma}_2)=\sigma_x(\overline{\Sigma}_1)\cap\sigma_x(\overline{\Sigma}_2)$ and both $\sigma_x(\overline{\Sigma}_1)$ and $\sigma_x(\overline{\Sigma}_2)$ do  by (b). 
\end{proof}
\begin{proof}[Proof of Theorem \ref{thm:affstrat:chamberdec:correspondence}] Let $X\subset V$ be as in the theorem. In view of Proposition \ref{prop:locpoldec}, it remains to prove that if $\mathcal{S}$ is a stratification of $X$ for which $(X,\mathcal{S})$ is locally modeled on fans, then the collection $\mathcal{P}(\mathcal{S})$ consisting of the closures in $X$ of strata of $\mathcal{S}$ is a locally polyhedral chamber decomposition of $X$ and the stratification associated to $\mathcal{P}(\mathcal{S})$ as in Proposition \ref{prop:locpoldec} equals $\mathcal{S}$. For this, note first that $\mathcal{P}(\mathcal{S})$ is a cover of $X$ consisting of locally polyhedral sets by part (a) of Lemma \ref{lemma:strataclosures:locfanstrat}. It is locally finite since $\mathcal{S}$ is so. Moreover, from part (c) of Lemma \ref{lemma:strataclosures:locfanstrat} it is immediate that $\mathcal{P}(\mathcal{S})$ satisfies condition (i) in Definition \ref{def:locallypolyhedralchamberdecomposition}. By combining part (d) of Lemma \ref{lemma:strataclosures:locfanstrat} with Lemma \ref{lemma:locpolstrat:regpartconnected} below, it follows that if $\Sigma_1,\Sigma_2\in \mathcal{S}$ are such that $\overline{\Sigma}_1$ and $\overline{\Sigma}_2$ intersect, then $\overline{\Sigma}_1\cap \overline{\Sigma}_2$ is the closure in $X$ of a unique stratum in $\mathcal{S}$, which by part (c) of Lemma \ref{lemma:strataclosures:locfanstrat} is a stratum of both $\mathcal{S}(\overline{\Sigma}_1)$ and $\mathcal{S}(\overline{\Sigma}_2)$. Hence, $\mathcal{P}(\mathcal{S})$ satisfies condition (ii) in Definition \ref{def:locallypolyhedralchamberdecomposition} as well. Finally, by part (c) of Lemma \ref{lemma:strataclosures:locfanstrat} the stratification associated to $\mathcal{P}(\mathcal{S})$ as in Proposition \ref{prop:locpoldec} equals $\mathcal{S}$.
\end{proof}
Here we used the following lemma.
\begin{lemma}\label{lemma:locpolstrat:regpartconnected} For any connected locally polyhedral set $P$ in $V$, there is a unique stratum of $\mathcal{S}(P)$ that is open and dense in $P$. 
\end{lemma}
\begin{proof} Let $R$ be the union of strata that are open in $P$. As for any stratified space, $R$ is dense in $P$. Therefore, it suffices to show that $R$ is connected. Since $R$ is dense in $P$ and $P$ is connected, to prove that $R$ is connected it is enough to show that any $x\in P$ admits an open neighbourhood that intersects $R$ in a connected subspace. For any $x\in P$ and any open $U$ in $V$ around $x$ such that $P\cap U=\sigma_x(P)\cap U$ the intersection of $U$ with $R$ is $(x+\mathring{\sigma}_x(P))\cap U$. So, for any convex such $U$, the open neighbourhood $P\cap U$ indeed intersects $R$ in a connected subspace.
\end{proof}
\section{The chamber decomposition of the momentum map image}\label{sec:chamberdecmomim} Whilst, for simplicity, we have restricted attention to actions on compact connected manifolds in the introduction, we will in fact prove the following more general versions of Theorem \ref{thm:folklore:intro} and Theorem \ref{thm:main:stratcomparison} in what follows. 
\begin{theorem}\label{thm:momim:locpoldec} Let $T$ be a torus and $\mu:(S,\omega)\to \mathfrak{t}^*$ a connected Hamiltonian $T$-space such that $\mu$ is proper as map into some convex subset of $\mathfrak{t}^*$. Then the connected components of the set of relative regular values attained by $\mu$ are convex open subsets of $\Delta$ and their closures in $\Delta$ are locally polyhedral sets. The strata of these locally polyhedral sets are convex and their closures form a locally polyhedral decomposition $\mathcal{P}_\mathrm{Ham}(\Delta)$ of $\Delta$ as in Definition \ref{def:locallypolyhedralchamberdecomposition}. 
\end{theorem}
\begin{theorem}\label{thm:momim:stratcomparison} The stratification $\mathcal{S}_\mathrm{Ham}(\Delta)$ in \cite{Mol} coincides with the stratification associated (as in Theorem \ref{thm:affstrat:chamberdec:correspondence}) to the locally polyhedral decomposition $\mathcal{P}_\mathrm{Ham}(\Delta)$ in Theorem \ref{thm:momim:locpoldec}.
\end{theorem}

Note here that the conditions required in the above theorems imply that the fibers of $\mu$ are connected and that $\Delta$ is itself convex (see, e.g. \cite{BjKa10}).\\

We will prove Theorem \ref{thm:momim:locpoldec} and Theorem \ref{thm:momim:stratcomparison} simultaneously, by showing that the affine stratification $\mathcal{S}_\mathrm{Ham}(\Delta)$ is locally modeled on fans and applying Theorem \ref{thm:affstrat:chamberdec:correspondence}. As somewhat of a by-product, this gives a slight alternative to the proof of \cite[Theorem 1.1]{Liu26} (the existence of which we became aware of only after completing the first version of this paper).

\subsection{Background on the stratification}\label{sec:bckgr:hamstrat} First, we recall the relevant properties and parts of the construction of $\mathcal{S}_\mathrm{Ham}(\Delta)$, starting with a general construction of affine stratifications. \\

Let $X$ be a locally closed subset of a finite-dimensional real vector space $V$. By a \textbf{piecewise-affine cover} of $X$ we mean a locally finite cover $\mathcal{A}$ of $X$ with the following properties.
\begin{itemize}
\item[(i)] Each member of $\mathcal{A}$ is relatively open in $V$.
\item[(ii)] The closure in $X$ of each member of $\mathcal{A}$ is a union of members of $\mathcal{A}$. 
\end{itemize}
The following is proved in \cite{Mol}.
\begin{proposition}\label{prop:piecewise-affine-cover} Let $\mathcal{A}=\{P_i\mid i\in I\}$ be a piecewise-affine cover of $X$. There is a unique stratification $\mathcal{S}(\mathcal{A})$ of $X$ into affine submanifolds of $V$ such that
\[
T_x\Sigma=\bigcap_{i\in I_x} T_xP_i,\quad I_x:=\{i\in I\mid x\in P_i\},
\] for each stratum $\Sigma\in \mathcal{S}(\mathcal{A})$ and each $x\in \Sigma$. 
\end{proposition}
The stratification of $\mathcal{S}_\mathrm{Ham}(\Delta)$ is obtained by applying this to the piecewise-affine cover $\mathcal{A}_\mathrm{Ham}$ consisting of the images of the orbit type strata of the $T$-action on $S$ under the momentum map $\mu$. As was already observed in \cite{GS82}, since $T$ is abelian the image $P=\mu(\Sigma_S)$ of any orbit type stratum $\Sigma_S\subset S$ is relatively open in $\mathfrak{t}^*$, with tangent space at a given point $x\in P$ equal to the annihilator $\mathfrak{t}_p^0$ of the isotropy Lie algebra at any point $p\in \mu^{-1}(x)\cap \Sigma_S$. So, given $\Sigma\in \mathcal{S}_\mathrm{Ham}(\Delta)$ and $x\in \Sigma$, its tangent space at $x$ is
\begin{equation}\label{eqn:tngtsp:hamstrat}
T_x\Sigma=\bigcap_{p\in \mu^{-1}(x)} \mathfrak{t}_p^0. 
\end{equation} By the same argument as in the proof of Proposition \ref{prop:affstrat:locmodfan:tangentcrit}, this characterizes the affine stratification $\mathcal{S}_\mathrm{Ham}(\Delta)$ uniquely. It is closely related to the so-called x-ray considered in \cite{Tol98}. 
\begin{example}\label{example:hamstrat} Consider the $\Tt^2$-action on $\mathbb{C}P^3$ given by \[
(\lambda_1,\lambda_2)\cdot[z_0:z_1:z_2:z_3]=[\lambda_1^{-1}z_0:\lambda_1z_1:\lambda_2z_2:\lambda_2^{-1}z_3].
\] With respect to a standard symplectic form on $\mathbb{C}P^3$, this action is Hamiltonian with momentum map 
\[
\mu([z_0:z_1:z_2:z_3])=\frac{1}{||z||^2}(-|z_0|^2+|z_1|^2,|z_2|^2-|z_3|^2).
\] The fixed points of this action are $[1:0:0:0]$, $[0:1:0:0]$, $[0:0:1:0]$ and $[0:0:0:1]$, which map to $(-1,0)$, $(1,0)$, $(0,1)$ and $(0,-1)$. The momentum map image $\Delta$ is the convex hull of these four points.  The images of the other $\Tt^2$-orbit type strata under $\mu$ are the relatively open line segments
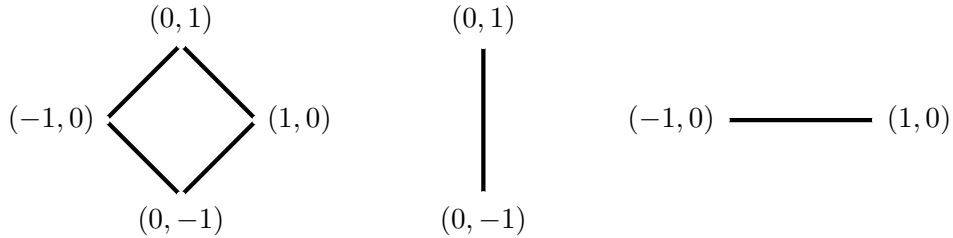
\begin{center}
\begin{tikzpicture} 

\draw[ultra thick,black] (-6,0) -- (-5,1) ;
\draw[ultra thick,black] (-5,1) -- (-4,0) ;
\draw[ultra thick,black] (-4,0) -- (-5,-1) ;
\draw[ultra thick,black] (-5,-1) -- (-6,0) ;
\draw[ultra thick,black] (-1,1) -- (-1,-1);
\draw[ultra thick,black] (2.2,0) -- (4.2,0);

\filldraw[white] (-6,0) circle (1.5pt) node[black,left] {$(-1,0)$} ;
\filldraw[white] (-4,0) circle (1.5pt) node[black, right] {$(1,0)$};
\filldraw[white] (-5,1) circle (1.5pt) node[black,above] {$(0,1)$};
\filldraw[white] (-5,-1) circle (1.5pt) node[black, below] {$(0,-1)$};
\filldraw[white] (-1,-1) circle (1.5pt) node[black, below] {$(0,-1)$};
\filldraw[white] (-1,1) circle (1.5pt) node[black, above] {$(0,1)$};
\filldraw[white] (2.2,0) circle (1.5pt) node[black, left] {$(-1,0)$};
\filldraw[white] (4.2,0) circle (1.5pt) node[black, right] {$(1,0)$};

\end{tikzpicture}
\end{center} and the interior of $\Delta$. The stratification $\mathcal{S}_\mathrm{Ham}(\Delta)$ therefore is as follows.
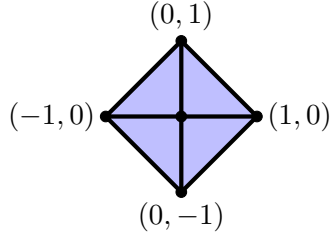
\begin{center}
\begin{tikzpicture} 
\filldraw[blue!25!white] (-1,0) -- (0,1) -- (1,0) -- (0,-1) -- cycle ;

\draw[ultra thick,black] (-1,0) -- (0,1);
\draw[ultra thick,black] (0,1) -- (1,0) ;
\draw[ultra thick,black] (1,0) -- (0,-1) ;
\draw[ultra thick,black] (0,-1) -- (-1,0) ;
\draw[ultra thick,black] (0,1) -- (0,-1);
\draw[ultra thick,black] (-1,0) -- (1,0);

\filldraw[black] (-1,0) circle (2pt) node[left] {$(-1,0)$};
\filldraw[black] (1,0) circle (2pt) node[right] {$(1,0)$};
\filldraw[black] (0,1) circle (2pt) node[ above] {$(0,1)$};
\filldraw[black] (0,-1) circle (2pt) node[below] {$(0,-1)$};
\filldraw[black] (0,0) circle (2pt) ;

\end{tikzpicture}
\end{center}
\end{example}
\subsection{Background on the GKZ fan of a vector configuration} In order to show that $\mathcal{S}_\mathrm{Ham}(\Delta)$ is locally modelled on fans, we will use the GKZ fan of a vector configuration. In this section we give a reminder on this, focussing on the properties that are relevant to us. Let $V$ be a finite-dimensional real vector space and let $\mathcal{V}$ be a \textbf{vector configuration} in $V$, meaning a finite collection of vectors in $V$. The \textbf{Gelfand-Kapranov-Zelevinski (GKZ) fan} $\mathrm{Ch}(\mathcal{V})$ (also called the chamber fan or secondary fan) of the vector configuration is a certain fan with support the polyhedral cone 
\[ 
\mathrm{Cone}(\mathcal{V}):=\left\{\sum_{v\in \mathcal{V}}s_vv\mid s_v\geq 0\right\}
\]
generated by $\mathcal{V}$. The relatively open chambers of $\mathrm{Ch}(\mathcal{V})$ are the members of the partition given by the equivalence relation
\[
x\sim y \iff \{\mathcal{I}\subset \mathcal{V}\mid x\in \textrm{Cone}(\mathcal{I})\}=\{\mathcal{I}\subset\mathcal{V} \mid y\in \textrm{Cone}(\mathcal{I})\}
\] on $\mathrm{Cone}(\mathcal{V})$, with $\textrm{Cone}(\mathcal{I})$ the polyhedral cone generated by the subset $\mathcal{I}$, and the closures of these chambers are the cones of $\textrm{Ch}(\mathcal{V})$. The relatively open chamber of $\mathrm{Ch}(\mathcal{V})$ through a point $x$ can also be described as the intersection of the relative interiors of the polyhedral cones spanned by all subsets $\mathcal{I}\subset \mathcal{V}$ for which $x$ lies in the relative interior of $\mathrm{Cone}(\mathcal{I})$ (see, e.g., \cite[Section 5.4]{DeLoRaSa}). In view of this and the fact that the relative interior of $\mathrm{Cone}(\mathcal{I})$ equals the strictly positive linear span of $\mathcal{I}$, the following holds.
\begin{lemma}\label{lemma:stratification:chamberfan} For any relatively open chamber $\mathring{\sigma}$ of $\mathrm{Ch}(\mathcal{V})$ and any  $x\in \mathring{\sigma}$, the tangent space to $\mathring{\sigma}$ at $x$ is given by
\[
T_x\mathring{\sigma}=\bigcap_{\{\mathcal{I}\subset \mathcal{V}\mid\text{ }x=\sum_{v\in \mathcal{I}} s_vv\,\mathrm{for\text{ }some}\,s_v>0\}} \mathrm{Span}_\R(\mathcal{I}). 
\] 
\end{lemma}

\subsection{The chamber decomposition for linear symplectic torus actions} The vector configurations that we will consider are those given by the weights of symplectic representations. To elaborate, recall that a \textbf{symplectic representation} of a Lie group $H$ is a symplectic vector space $(V,\omega_V)$ equipped with a linear symplectic $H$-action. Such an $H$-action is  Hamiltonian, with momentum map $\mu_V:V\to \mathfrak{h}^*$ given by
\begin{equation}\label{eqn:quadmommap:linsymprep}
\langle \mu_{V}(v),\xi\rangle:=\frac{1}{2}\omega_V(\xi\cdot v,v),\quad \xi\in \mathfrak{h},\quad v\in V.
\end{equation} 
The symplectic representations of a torus $T$ are built out of the irreducible symplectic representations 
\[
(\C_{\alpha},\omega_\mathrm{st}),\quad \alpha\in \Lambda_T^*,
\] which are indexed by the lattice $\Lambda_T\subset \mathfrak{t}^*$ dual to $\Lambda_T:=\ker(\exp_T)$ and consist of the $\R$-vector space $\C$ with the linear symplectic form $\omega_\mathrm{st}$ given by
\[
\omega_\mathrm{st}(w,z)=\frac{1}{2\pi i}(w\overline{z}-\overline{w}z)\in \R, \quad w,z\in \C,
\] and with the $T$-action given by
\[
\exp_T(\xi)\cdot z:=e^{2\pi i\langle\alpha,\xi\rangle}z,\quad \xi\in \mathfrak{t},\quad z\in \C.  
\]
That is, for any symplectic $T$-representation $(V,\omega_V)$ there is a unique unordered tuple $(\alpha_1,...,\alpha_n)$ -- called its \textbf{weight tuple} -- such that $(V,\omega_V)$ is isomorphic to the symplectic $T$-representation $(\C_{\alpha_1},\omega_\mathrm{st})\oplus ...\oplus (\C_{\alpha_n},\omega_\mathrm{st})$. Via such an isomorphism, the momentum map $\mu_V$ is identified with the map
\begin{equation}\label{eqn:weight:description:quadmommap}
(z_1,...,z_n)\mapsto \sum_{i=1}^n|z_i|^2\alpha_i.
\end{equation}

Therefore, the image $\Delta_V$ of $\mu_V$ is the cone generated by the weight tuple and, hence, the GKZ fan of the set of weights of $(V,\omega_V)$, that we will denote as $\F_V$, is a polyhedral fan with support $\Delta_V$. This fan is the linear analogue of the  decomposition in Theorem \ref{thm:momim:locpoldec}. The tangent spaces to the relatively open cones of $\mathcal{F}_V$ can be characterized as follows, akin to \eqref{eqn:tngtsp:hamstrat}.
\begin{proposition}\label{prop:linearchamberdec} Let $(V,\omega_V)$ be a symplectic representation of a torus $T$, $x\in \Delta_V$ and $\mathring{\sigma}$ the relatively open cone of $\F_V$ through $x$. Then
\begin{equation*} T_x\mathring{\sigma}=\bigcap_{p\in \mu_V^{-1}(x)} \mathfrak{t}_p^0.
\end{equation*}
\end{proposition}
\begin{proof} We can suppose  that $(V,\omega_V)=(\C_{\alpha_1},\omega_\mathrm{st})\oplus ...\oplus (\C_{\alpha_n},\omega_\mathrm{st})$. From the description \eqref{eqn:weight:description:quadmommap} of $\mu_V$ and the fact that $\mathfrak{t}_p^0=\im(\d \mu_V)_p$ it follows that
\[
\mathfrak{t}_p^0=\mathrm{Span}_\R\{\alpha_i\mid z_i\neq 0\},\quad  (z_1,...,z_n):=p,
\] for all $p\in V$. On the other hand, by Lemma \ref{lemma:stratification:chamberfan} it holds that
\[
T_x\mathring{\sigma}=\bigcap_{\big\{I\subset \{1,...,n\}\mid \text{ }x=\sum_{i\in I}t_i\alpha_i\text{ for some }  t_i>0\big\}} \mathrm{Span}_\R\{\alpha_i\mid i\in I\}.
\] Since $x=\sum_{i\in I}t_i\alpha_i$ for some $t_i>0$ if and only if $x=\mu_V(p)$ for some $p=(z_1,...,z_n)$ such that $I=\{i\mid z_i\neq 0\}$, the proposition follows.  
\end{proof} 

\subsection{The isotropy fans and proof of the main theorems} Let $T$ be a torus and $\mu:(S,\omega)\to \mathfrak{t}^*$ a Hamiltonian $T$-space such that $\mu$ has connected fibers and is proper as map into its image $\Delta:=\mu(S)$. The momentum maps in Theorem \ref{thm:momim:locpoldec} and Theorem \ref{thm:momim:stratcomparison} have both of these properties (see \cite[Theorem 30]{BjKa10} for connected of the fibers). We will now construct a natural fan $\mathcal{F}_x$ for each $x\in \Delta$. For this, recall that any point $p\in S$ has a naturally associated symplectic representation $(\mathcal{S}\mathcal{N}_p,\omega_p)$ of the isotropy group $T_p$, called the \textbf{symplectic normal (or slice) representation}. This consists of the vector space
\begin{equation*}
\mathcal{S}\mathcal{N}_p:=\frac{T_p\O^\omega}{T_p\O\cap T_p\O^\omega}
\end{equation*} (the fiber of the symplectic normal bundle to the orbit $\O$ through $p$), equipped with linear symplectic form induced by $\omega$ and the linear symplectic $T_p$-action given by 
\begin{equation}\label{eqn:sympnormrep:action}
t\cdot [v]=[(\d m_t)_p(v)], \quad t\in T_p, \quad v\in T_p\O^\omega,
\end{equation} where $m_t:S\to S$ denotes the action of $t$. Since $T_p$  is compact and abelian, its identity component $T_p^0$ is a torus. Therefore, the induced symplectic $T_p^0$-representation has an associated set of weights and an associated fan 
\[
\F_p:=\F_{\mathcal{S}\mathcal{N}_p}.
\] The momentum map of the induced symplectic $T^0_p$-action is the same as that of the $T_p$-action. So, the support of $\F_p$ (the cone generated by the weights) is the image of  $\mu_{\mathcal{S}\mathcal{N}_p}:{\mathcal{S}\mathcal{N}_p}\to \mathfrak{t}_p^*$ (see \eqref{eqn:quadmommap:linsymprep}), which we denote by $\Delta_p$. Furthermore, the following holds.
\begin{proposition}\label{prop:associatedfan:ind:orbstrat} For any two $p,q\in S$ in the same $T$-orbit type stratum, it holds that $\F_p=\F_q$. 
\end{proposition}
\begin{proof} As must be well-known, for all such points there is an isomorphism of symplectic representations $(\mathcal{S}\mathcal{N}_p,\omega_p)\cong (\mathcal{S}\mathcal{N}_q,\omega_q)$ of $T_p=T_q$. This can be seen by using the MGS normal form theorem \cite{GS84,Mar85} to show that, for any $T$-orbit type stratum $\Sigma$, all members of the equivalence relation
\[
p\sim q \iff (\mathcal{S}\mathcal{N}_p,\omega_p)\cong (\mathcal{S}\mathcal{N}_q,\omega_q)\text{ as symplectic representations of $T_p=T_q$}
\] on $\Sigma$ are open in $\Sigma$, so that there can be only one such member because $\Sigma$ is connected. The proposition is immediate from this. 
\end{proof} 
Because of this proposition, the fan $\mathcal{F}_p$ is the same for all points $p$ in a given $T$-orbit type stratum $\Sigma$. Therefore we also denote this fan by $\F_\Sigma$. Similarly, we denote by $\mathfrak{t}_\Sigma$ the isotropy Lie algebra of any point in $\Sigma$ and, accordingly, we let $\pi_\Sigma:\mathfrak{t}^*\to \mathfrak{t}^*_\Sigma$ denote the projection dual to the inclusion $\mathfrak{t}_\Sigma\hookrightarrow \mathfrak{t}$. Since $\mu$ is proper as map into its image, its fibers are compact and, hence, each of its fibers intersects only finitely many strata of the $T$-orbit type stratification $\mathcal{S}_T(S)$ of $S$. So, we can define the following.  
\begin{definition}\label{def:isofans} For a Hamiltonian $T$-space as above, we define its \textbf{isotropy fan} at a point $x$ in the momentum map image to be the fan in $\mathfrak{t^*}$ given by
\[
\F_x:=\bigcap_{\big\{\Sigma\in \mathcal{S}_T(S)\mid \Sigma\cap \mu^{-1}(x)\neq \emptyset\big\}} \pi_\Sigma^{-1}(\F_\Sigma). 
\] 
\end{definition} Here, the pre-image of a fan is meant cone-wise and the intersection of fans is meant in the sense of the following elementary lemma. 
\begin{lemma}\label{lemma:intersection:fans}
    Given polyhedral fans $\mathcal{F}_1$,...,$\mathcal{F}_n$ in a finite-dimensional real vector space $V$, their intersection
    \[
    \mathcal{F}_1\cap...\cap\mathcal{F}_n:=\{\sigma_1\cap...\cap\sigma_n\mid \sigma_i\in \mathcal{F}_i\}
    \] is again a polyhedral fan. Moreover, \[\mathrm{supp}(\mathcal{F}_1\cap...\cap\mathcal{F}_n)=\mathrm{supp}(\mathcal{F}_1)\cap...\cap\mathrm{supp}(\mathcal{F}_n),\] and each relatively open cone of $\mathcal{F}_1\cap ...\cap{\F}_n$ is of the form $\mathring{\sigma}_1\cap ...\cap \mathring{\sigma}_n$ with $\mathring{\sigma}_i$ a relatively open cone of $\mathcal{F}_i$.
\end{lemma}
In view of this lemma and the proposition below, $\Delta$ is a locally polyhedral set and the support of $\F_{x_0}$ is the cone $\sigma_{x_0}(\Delta)$. So, for each $x_0\in \Delta$, the germ of $\Delta$ at $x_0$ equals that of $x_0+\mathrm{Supp}(\mathcal{F}_{x_0})$.
\begin{proposition}\label{prop:cone:mommapimage} The image $\Delta$ of $\mu$ is a locally polyhedral subset of $\mathfrak{t}^*$. Given $x\in \Delta$, the cone $\sigma_x(\Delta)$ on which $\Delta$ is modelled near $x$ equals $\pi_p^{-1}(\Delta_p)$ for any choice of $p\in \mu^{-1}(x)$, where $\pi_p:\mathfrak{t}^*\to \mathfrak{t}^*_p$ is the projection dual to the inclusion $\mathfrak{t}_p\hookrightarrow \mathfrak{t}$.
\end{proposition}
This well-known proposition holds because of the topological assumptions that we made on $\mu$. A proof is readily obtained by adapting the arguments in \cite{BiOrRa09,BjKa10, HiNePl93}. Having defined the isotropy fans, we now prove the following and then complete the proof of Theorem \ref{thm:momim:locpoldec} and Theorem \ref{thm:momim:stratcomparison}. 
\begin{proposition}\label{prop:isofans:locmodhamstrat} For $\mu:(S,\omega)\to \mathfrak{t}^*$ and $\Delta$ as above, the stratified subset $(\Delta,\mathcal{S}_\mathrm{Ham}(\Delta))$ of $\mathfrak{t}^*$ is locally modelled on the isotropy fans in Definition \ref{def:isofans}.
\end{proposition}
To prove this proposition, we will use the following lemma.
\begin{lemma}\label{lemma:affstratlocconvnhood} Let $V$ be a finite-dimensional real vector space, $X$ a locally convex subset of $V$ and $\mathcal{S}$ an affine stratification of $X$. Then every $x_0\in X$ admits an arbitrarily small convex open neighbourhood $U$ in $X$ such that for every $x\in U$ the half-open line segment
\[
]x_0,x]:=\{x_0+s(x-x_0)\mid s\in ]0,1]\}
\] is contained in the stratum of $\mathcal{S}$ through $x$. 
\end{lemma}
\begin{proof}[Proof of Lemma \ref{lemma:affstratlocconvnhood}] Let $x_0\in X$ and let $W$ be an open in $X$ around $x_0$. Since $\mathcal{S}$ is locally finite, there is an open $U$ in $W$ such that $x_0\in \overline{\Sigma}$ for each stratum $\Sigma$ that intersects $U$. Because $X$ is locally convex, after shrinking $U$ we can arrange it to be convex. Let $x\in U$. Consider the map $\gamma:]0,1]\to X$, $\gamma(s):=x_0+s(x-x_0)$. To prove the lemma, we will show that $\gamma^{-1}(\Sigma)$ is open in $]0,1]$ for each stratum $\Sigma$, so that by connectedness of $]0,1]$ the image $]x_0,x]$ of $\gamma$ must be contained in a single such stratum. To this end, let $\Sigma\in \mathcal{S}$. We can suppose that $\gamma^{-1}(\Sigma)\neq\emptyset$. Then there is an $s\in ]0,1]$ such that $\gamma(s)\in \Sigma$. Since, moreover, $\gamma(s)\in U$ by convexity of $U$, the stratum $\Sigma$ intersects $U$ and, so, $x_0\in \overline{\Sigma}$. Since $\mathcal{S}$ is affine, $\Sigma$ is open in its affine hull $A$. Since $A$ is affine, it is closed in $V$ and so $\overline{\Sigma}\subset A$. Therefore, both $x_0$ and $\gamma(s)$ are contained in $A$. Since $A$ is an affine subset of $V$ and $s>0$, it follows that, in fact, $]x_0,x]\subset A$. So, $\gamma$ defines a map into $A$. Since this is continuous and $\Sigma$ is open in $A$, it follows that $\gamma^{-1}(\Sigma)$ is open in $]0,1]$, as was to be shown. 
\end{proof}
\begin{proof}[Proof of Proposition \ref{prop:isofans:locmodhamstrat}] Let $x_0\in \Delta$. As discussed after Lemma \ref{lemma:intersection:fans}, the germ of $\Delta$ at $x_0$ equals that of $x_0+\mathrm{Supp}(\mathcal{F}_{x_0})$. Because of this, Proposition \ref{prop:affstrat:locmodfan:tangentcrit} and Proposition \ref{prop:piecewise-affine-cover}, it is enough to show that $x_0$ admits an open neighbourhood $U$ such that for each relatively open chamber $\mathring{\sigma}$ of $\F_{x_0}$ and each $x\in (x_0+\mathring{\sigma})\cap U$ the tangent space to $x_0+\mathring{\sigma}$ at $x$ is given by
\begin{equation}\label{eqn:locmodfan:mommap:0}
T_x(x_0+\mathring{\sigma})=\bigcap_{p\in \mu^{-1}(x)} \mathfrak{t}_p^0.
\end{equation}
To this end note that, since $\mu$ is proper as map onto its image, there is a convex open $U$ in $\Delta$ around $x_0$ that satisfies the property in Lemma \ref{lemma:affstratlocconvnhood} with respect to the affine stratification $\mathcal{S}_\mathrm{Ham}(\Delta)$, and for which there are $p_1,...,p_n\in \mu^{-1}(x_0)$ and $T$-invariant open $\widehat{U}_i$ around $p_i$ in $S$ such that
\begin{itemize}
    \item[(i)] on each $\widehat{U}_i$ there is an isomorphism of Hamiltonian $T$-spaces $\psi_i$ that transforms the Hamiltonian $T$-space $\mu$ into its MGS local normal form around the $T$-orbit through $p_i$ \cite{GS84,Mar85},
\item[(ii)] $\mu^{-1}(U)\subset \bigcup_{i=1}^n\widehat{U}_i$,
\item[(iii)] for each $\Sigma\in \mathcal{S}_T(S)$ that intersects $\mu^{-1}(x_0)$ some $p_i$ lies in $\Sigma$.
 
\end{itemize} We will show that \eqref{eqn:locmodfan:mommap:0} holds for this open $U$. Consider a relatively open chamber $\mathring{\sigma}$ of $\mathcal{F}_{x_0}$ and some $x\in (x_0+\mathring{\sigma})\cap U$. For $\Sigma\in \mathcal{S}_T(S)$ such that $\Sigma\cap\mu^{-1}(x_0)\neq\emptyset$, let $\mathring{\sigma}_\Sigma$ denote the relatively open cone of $\mathcal{F}_\Sigma$ that contains $\pi_\Sigma(x-x_0)$. In view of Lemma \ref{lemma:intersection:fans}, $\mathring{\sigma}$ is the intersection of these relatively open cones. Because of this and (iii) above, 
\begin{equation*}
T_x\big(x_0+\mathring{\sigma})=\bigcap_{\{\Sigma\in\mathcal{S}_T(S)\mid \Sigma\cap \mu^{-1}(x_0)\neq\emptyset\}} T_x\big(x_0+\pi_{\Sigma}^{-1}(\mathring{\sigma}_\Sigma)\big) =\bigcap_{i=1}^n T_x\big(x_0+\pi_{p_i}^{-1}(\mathring{\sigma}_i)\big),
\end{equation*}
where $\mathring{\sigma}_i$ denotes the relatively open cone of $\mathcal{F}_{p_i}$ that contains $\pi_{p_i}(x-x_0)$. 
Further note that, in view of Proposition \ref{prop:linearchamberdec},
\[
T_x\big(x_0+\pi_{p_i}^{-1}(\mathring{\sigma}_i)\big)=\pi_{p_i}^{-1}\big(T_{\pi_{p_i}(x-x_0)}\mathring{\sigma_i}\big)=\bigcap_{v\in \mu^{-1}_{\mathcal{SN}_{p_i}}(\pi_{p_i}(x-x_0))} (\mathfrak{t}_{p_i})_v^0
\] for each $p_i$, where the annihilator of the isotropy Lie algebra $(\mathfrak{t}_{p_i})_v$ of the $T_{p_i}$-action on $\mathcal{SN}_{p_i}$ is that in $\mathfrak{t}^*$. Using (i) above we will now  show that
\begin{equation}\label{eqn:intersectionequality}
\bigcap_{v\in \mu^{-1}_{\mathcal{SN}_{p_i}}(\pi_{p_i}(x-x_0))} (\mathfrak{t}_{p_i})_v^0=\bigcap_{p\in \widehat{U}_i \cap \mu^{-1}(]x_0,x])} \mathfrak{t}_p^0. 
\end{equation} For this, recall that the isomorphism $\psi_i$ in (i) above (in particular) is a $T$-equivariant diffeomorphism from $\widehat{U}_i$ onto a $T$-invariant open neighbourhood of $([1,0],0)$ in $ (T\times_{T_{p_i}} \mathcal{SN}_{p_i})\times \mathfrak{t}_{p_i}^0$. Here  $T\times_{T_{p_i}} \mathcal{SN}_{p_i}$ is the quotient of $T\times\mathcal{SN}_{p_i}$ under the diagonal $T_{p_i}$-action (where $T_{p_i}$ acts on $T$ by translation and on $\mathcal{SN}_{p_i}$ as in \eqref{eqn:sympnormrep:action}) equipped with the $T$-action on $(T\times_{T_{p_i}} \mathcal{SN}_{p_i})\times \mathfrak{t}_{p_i}^0$ given by $t\cdot ([\tau,v],\alpha)=([t\tau,v],\alpha)$. Under $\psi_i$, the momentum map $\mu$ becomes identified with (the restriction to the image of $\psi_i$ of) a map of the form
\[
\mu_i: (T\times_{T_{p_i}} \mathcal{SN}_{p_i})\times \mathfrak{t}_{p_i}^0\to \mathfrak{t}^*, \quad ([t,v],\alpha)\mapsto x_0+\alpha+\mathfrak{p}_i\big(\mu_{\mathcal{SN}_{p_i}}(v)\big),
\] where $\mu_{\mathcal{SN}_{p_i}}$ is as in \eqref{eqn:quadmommap:linsymprep} and $\mathfrak{p}_i:\mathfrak{t}^*_{p_i}\to \mathfrak{t}^*$ is a linear section of $\pi_{p_i}$. Using this it follows that if $\beta$ belongs to the left hand side of \eqref{eqn:intersectionequality}, then for any $p\in \widehat{U}_i$ with $\mu(p)\in ]x_0,x]$, writing $\psi(p)=([\tau,v],\alpha)$ it holds that $\mu_i([\tau,v],\alpha)=x_0+s(x-x_0)$ for some $s\in ]0,1]$ and, hence, $\mu_{\mathcal{SN}_{p_i}}(\tfrac{1}{\sqrt{s}}v)=\pi_{p_i}(x-x_0)$. Since $(T_{p_i})_{\delta v}=(T_{p_i})_v$ for all $\delta\neq 0$ (by linearity of the $T_{p_i}$-action on $\mathcal{SN}_{p_i}$) and $T_{([\tau,\alpha],v)}=(T_{p_i})_v$, it follows that  
\[
\beta\in (\mathfrak{t}_{p_i})^0_{\tfrac{1}{\sqrt{s}}v}=(\mathfrak{t}_{p_i})^0_v=\mathfrak{t}^0_{([\tau,v],\alpha)}=\mathfrak{t}_p^0,
\] which shows that $\beta$ belongs to the right-hand side of \eqref{eqn:intersectionequality}. Conversely, suppose $\beta$ belongs to the right-hand side of \eqref{eqn:intersectionequality} and  $v\in \mu^{-1}_{\mathcal{SN}_{p_i}}(\pi_{p_i}(x-x_0))$. Consider $s\in ]0,1]$ small enough so that $([1,sv],s^2\alpha)$ lies in the open $\psi(\widehat{U}_i)$ around $([1,0],0)$, with $\alpha\in\mathfrak{t}_{p_i}^0=\ker(\pi_{p_i})$ given by $x-x_0=\alpha+\mathfrak{p}_i(\pi_{p_i}(x-x_0))$. For $p:=\psi_i^{-1}([1,sv],s^2\alpha)$ it holds that $\mu(p)=x_0+s^2(x-x_0)\in ]x_0,x]$. So, 
\[
\beta\in \mathfrak{t}_p^0=\mathfrak{t}_{([1,sv],s^2\alpha)}^0=\mathfrak{t}_{v}^0, 
\] 
which shows that $\beta$ belongs to the left-hand side of \eqref{eqn:intersectionequality}. Thus, \eqref{eqn:intersectionequality} holds. Combining this with (ii) and what was shown before \eqref{eqn:intersectionequality}, we find that 
\[
T_x\big(x_0+\mathring{\sigma})=\bigcap_{p\in \mu^{-1}(]x_0,x])} \mathfrak{t}_p^0.
\] To conclude from this that \eqref{eqn:locmodfan:mommap:0} holds, note that in view of \eqref{eqn:tngtsp:hamstrat}, the fact that $]x_0,x]$ is contained in the  stratum $\Sigma$ of $\mathcal{S}_\mathrm{Ham}(\Delta)$ through $x$, and the fact that $T_y\Sigma=T_x\Sigma$
 for all $y\in \Sigma$ ($\Sigma$ being relatively open), it holds that 
 \[
 \bigcap_{p\in \mu^{-1}(y)}\mathfrak{t}_p^0=\bigcap_{p\in \mu^{-1}(x)}\mathfrak{t}_p^0
 \] for all $y\in ]x_0,x]$.   
\end{proof}

\begin{proof}[Proof of Theorem \ref{thm:momim:locpoldec} and Theorem \ref{thm:momim:stratcomparison}] Since $\Delta$ is convex, any non-empty open subset of $\Delta$ has the same affine hull as $\Delta$. Using this, \eqref{eqn:tngtsp:hamstrat}, the fact that $\mathcal{S}_\mathrm{Ham}(\Delta)$ is affine, and the fact that the affine hull of $\Delta$ is  parallel to the annihilator of the intersection of all isotropy Lie algebras of the $T$-action, it readily follows that the strata of $\mathcal{S}_\mathrm{Ham}(\Delta)$ that are open in $\Delta$ are the connected components of the set of relative regular values attained by $\mu$. In view of Remark \ref{rem:maxelts:locpolchambdec} and Remark \ref{rem:maxelts=closuresopenstrata}, the theorems therefore follow by applying Theorem \ref{thm:affstrat:chamberdec:correspondence} and Theorem \ref{thm:convexity:strata:general} to $\mathcal{S}_\mathrm{Ham}(\Delta)$, which is possible because of Proposition \ref{prop:isofans:locmodhamstrat}. 
\end{proof}

\subsection{The relationship with $T_\mathbb{C}$-orbit closures in K\"{a}hler manifolds}\label{sec:orbitclosures} Suppose now that $(S,\omega)$ is a compact K\"{a}hler manifold with a Hamiltonian $T$-action by biholomorphisms. In that case, the $T$-action extends uniquely to a holomorphic action on $S$ of the complexification $T_\mathbb{C}$. For each $p\in S$, the image of the $T_\mathbb{C}$-orbit closure $\overline{T_\mathbb{C}\cdot p}$ in $S$ under the momentum map $\mu:(S,\omega)\to \mathfrak{t}^*$ is a convex polytope \cite[Theorem 2(a)]{At82}. Let $\mathcal{C}$ be the finite collection of these polytopes. This covers $\Delta:=\mu(S)$, is closed under taking faces, and has a partial order defined as
\[
P\leq Q \iff P\text{ is a face of }Q.
\]
As in \cite{GorMac87}, consider the partition of $\Delta$ given by the equivalence relation 
\[
x\sim y \iff \{P\in \mathcal{C}\mid x\in P\}=\{P\in \mathcal{C}\mid y\in P\}.
\] This partition, which we will denote as $\mathcal{P}_{T_\mathbb{C}}(\Delta)$, is in fact a stratification with possibly disconnected components. The member of $\mathcal{P}_{T_\mathbb{C}}(\Delta)$ that contains a given point $x\in \Delta$ is equal to the relatively open subset
\begin{equation}\label{eqn:member:part:GorMac}
\bigcap_{\{P\in \mathcal{C}\mid x\in \mathring{P}\}} \mathring{P}\quad \big\backslash\,\,\, \bigcup_{\{P\in \mathcal{C}\mid x\notin P\}} P, 
\end{equation} 
where the intersection is over all $P\in \mathcal{C}$ with $x$ in their relative interior $\mathring{P}$, or equivalently, all $P\in \mathcal{C}$ that contain $x$ and are minimal amongst such elements with respect to the above partial order on $\mathcal{C}$. As stated in the introduction, the following holds.
\begin{proposition}\label{prop:equivalence:GorMac:Hamstrat} The partition of $\Delta$ by connected components of the members of $\mathcal{P}_{T_\mathbb{C}}(\Delta)$ coincides with $\mathcal{S}_\mathrm{Ham}(\Delta)$. 
\end{proposition}
\begin{proof} In view of the description \eqref{eqn:member:part:GorMac}, the connected components of the members of $\mathcal{P}_{T_\mathbb{C}}(\Delta)$ partition $\Delta$ into relatively open connected subsets of $\mathfrak{t}^*$ with tangent spaces given by
\[ 
T_x\Sigma=\bigcap_{\{P\in \mathcal{C}\mid x\in \mathring{P}\}} T_x\mathring{P} 
\] 
for any such component $\Sigma$ and $x\in \Sigma$. Since $\mathcal{S}_\mathrm{Ham}(\Delta)$ is a partition of $\Delta$ into relatively open connected subsets as well, by the argument in for the proof of Proposition \ref{prop:affstrat:locmodfan:tangentcrit} it is enough to show that these tangent spaces are the same as those of the strata of $\mathcal{S}_\mathrm{Ham}(\Delta)$, which are given by \eqref{eqn:tngtsp:hamstrat}. So, we need to show that, for any $x\in \Delta$,
\begin{equation}\label{eqn:tgntsp:GorMac}
\bigcap_{\{P\in \mathcal{C}\mid x\in \mathring{P}\}}T_x\mathring{P}=\bigcap_{p\in \mu^{-1}(x)} \mathfrak{t}_p^0.
\end{equation}
 Given $p\in S$, for $P:=\mu(\overline{T_\mathbb{C}\cdot p})$ it holds that $\mathring{P}=\mu(T_\mathbb{C}\cdot p)$, in view of \cite[Theorem 2(b)]{At82}. Writing $\mathrm{a}_q:\mathfrak{t}\to T_qS$ for the linear map given by the Lie algebra action associated to the $T$-action, the tangent space to $T_\mathbb{C}\cdot p$ at a point $q$ is equal to $\mathrm{a}_q(\mathfrak{t})+J\mathrm{a}_q(\mathfrak{t})$, with $\mu$ the endomorphism of $TS$ defining the complex structure on $S$. Further note that $\mathrm{d}\mu_q\circ \mathrm{a}_q=0$ (because $\mu$ is $T$-invariant) and $\mathrm{d}\mu_q\circ J\circ \mathrm{a}_q$ descends to an injective map $\mathfrak{t}/\mathfrak{t}_q\to \mathfrak{t}^*$ with $\mathfrak{t}_q:=\ker(\mathrm{a}_q)$ the $\mathfrak{t}$-isotropy Lie algebra at $q$, for if $\xi\in \mathfrak{t}$ and $\xi\notin\mathfrak{t}_q$ then by the momentum map condition and positive definiteness of $\omega(\cdot,J\cdot)$, 
\[
\langle (\mathrm{d}\mu_q\circ J\circ \mathrm{a}_q)(\xi), \xi\rangle =\omega(\mathrm{a}_q(\xi),J\mathrm{a}_q(\xi))>0.
\] Since the image of $\mathrm{d}\mu_q$ is $\mathfrak{t}_q^0$, it therefore follows from a dimension count that $\mathrm{d}\mu_q$ maps the tangent space of $T_\mathbb{C}\cdot p$ at $q$ onto $\mathfrak{t}_q^0$. Because $T_\mathbb{C}$ is abelian, all points in $T_\mathbb{C}\cdot p$ have the same $\mathfrak{t}$-isotropy Lie algebra, and so $\mu:T_\mathbb{C}\cdot p\to \mathfrak{t}^*$ is a smooth map whose derivative at any point $q$ has image $\mathfrak{t}_p^0$. It must therefore be a submersion into an affine subspace parallel to $\mathfrak{t}_p^0$ and, hence, its image $\mathring{P}$ is open in this affine subspace. From this \eqref{eqn:tgntsp:GorMac} readily follows. 
\end{proof}
Recall that, for any $x\in \Delta$, the set $X^\mathrm{ss}_x$ of semistable points in $S$ with respect to the shifted momentum map $\mu-x$ is given by
\[
X^\mathrm{ss}_x:=\{p\in S\mid \overline{T_\mathbb{C}\cdot p}\cap \mu^{-1}(x)\neq\emptyset\}=\{p\in S\mid x\in \mu(\overline{T_\mathbb{C}\cdot p})\}.
\] By the analytic version of the Kempf-Ness theorem \cite{HeiLoo94,Kir84,Sja95}, the topological quotient $X^\mathrm{ss}_x\sslash T_\mathbb{C}$ of $X^\mathrm{ss}_x$ under the equivalence relation  
\[
p\sim q\iff \overline{T_\mathbb{C}\cdot p}\cap\overline{T_\mathbb{C}\cdot q}\cap X^\mathrm{ss}_x\neq\emptyset 
\] is naturally a complex analytic space and the inclusion of $\mu^{-1}(x)$ into $X^\mathrm{ss}_x$ descends to a homeomorphism 
\begin{equation}\label{eqn:kempfness:analytic}
\mu^{-1}(x)/T\cong X^\mathrm{ss}_x\sslash T_\mathbb{C}. 
\end{equation}
So, since the partition $\mathcal{P}_{T_\mathbb{C}}(\Delta)$ is also given by the equivalence relation 
\[
x\sim y\iff X^\mathrm{ss}_x=X^\mathrm{ss}_y,
\] the reduced spaces \eqref{eqn:kempfness:analytic} at any two values in the same member of $\mathcal{P}_{T_\mathbb{C}}(\Delta)$ are isomorphic as complex analytic spaces. It follows from Proposition \ref{prop:equivalence:GorMac:Hamstrat} that this also holds for any two values in the same stratum of $\mathcal{S}_\mathrm{Ham}(\Delta)$. 
\\

For disconnected $S$, the members of $\mathcal{P}_{T_\mathbb{C}}(\Delta)$ need not be connected. The following shows that they are connected when $S$ is connected.
\begin{proposition}\label{prop:partitionmember:intersection}
    If $S$ is connected, then for any $x\in \Delta$ the member of $\mathcal{P}_{T_\mathbb{C}}(\Delta)$ that contains $x$ is equal to the convex relatively open subset
    \begin{equation}\label{eqn:polytopepartition:connected}
    \bigcap_{\{P\in \mathcal{C}\mid x\in \mathring{P}\}} \mathring{P}.
    \end{equation}
\end{proposition} 
\begin{proof} Let $x\in \Delta$. Since the member containing $x$ is given by \eqref{eqn:member:part:GorMac}, it is an open subset of the connected set \eqref{eqn:polytopepartition:connected}. So, it is enough to show that \eqref{eqn:member:part:GorMac} is also closed in \eqref{eqn:polytopepartition:connected}. To this end, suppose that $y$ lies in the closure of \eqref{eqn:member:part:GorMac} in \eqref{eqn:polytopepartition:connected}. We ought to show $y$ does not lie in any $Q\in \mathcal{C}$ for which $x\not\in Q$. For this, suppose that $Q\in \mathcal{C}$ contains $y$. Then there is a $q\in \mu^{-1}(y)$ such that $\mu(T_\mathbb{C}\cdot q)\subset Q$. As noted in \cite{GorMac87}, the fact that  $y$ lies in the closure of \eqref{eqn:member:part:GorMac} implies that $X^\mathrm{ss}_x\subset X^\mathrm{ss}_y$ and the inclusion map descends via \eqref{eqn:kempfness:analytic} to a continuous (in fact, complex analytic) map 
\begin{equation}\label{eqn:specializationmap}
\mu^{-1}(x)/T\to \mu^{-1}(y)/T,
\end{equation} which sends the $T$-orbit of a point $p\in \mu^{-1}(x)$ to the $T$-orbit $\overline{T_\mathbb{C}\cdot p}\cap \mu^{-1}(y)$. Since $S$ is compact and connected, $X^\mathrm{ss}_x$ is dense in $S$ (see \cite{Kir84}) and, hence, $X^\mathrm{ss}_x$ is dense in $X^\mathrm{ss}_y$ too. So, since the image of \eqref{eqn:specializationmap} is that of $X^\mathrm{ss}_x$ under the quotient map from $X^\mathrm{ss}_y$ to $\mu^{-1}(y)/T$, it is dense as well. Because the domain of \eqref{eqn:specializationmap} is compact and its target is Hausdorff, its image is also closed. Therefore \eqref{eqn:specializationmap} is surjective and, so, there is a $p\in \mu^{-1}(x)$ such that $q$ is contained in the $T$-orbit $\overline{T_\mathbb{C}\cdot p}\cap \mu^{-1}(y)$. Since, by assumption, $y$ belongs to \eqref{eqn:polytopepartition:connected}, it lies in $\mathring{P}=\mu(T_\mathbb{C}\cdot p)$ for $P:=\mu(\overline{T_\mathbb{C}\cdot p})$. So, because $\mu^{-1}(\mathring{P})\cap \overline{T_\mathbb{C}\cdot p}=T_\mathbb{C}\cdot p$ (in view of \cite[Theorem 2(b)]{At82}), it must be that $q\in T_\mathbb{C}\cdot p$. Hence, $p\in T_\mathbb{C}\cdot q$ and, so, $x=\mu(p)\in Q$. 
\end{proof} 
In view of Proposition \ref{prop:equivalence:GorMac:Hamstrat}, this implies the following about $\mathcal{S}_\mathrm{Ham}(\Delta)$, which is a rephrasing of Proposition \ref{prop:orbitclosurepolytopesintersection}. 
\begin{corollary} If $S$ is connected, then the stratum of $\mathcal{S}_\mathrm{Ham}(\Delta)$ through a given point $x$ equals the intersection \eqref{eqn:polytopepartition:connected}. 
\end{corollary} 
\begin{remark} In contrast, it is not true that the strata of $\mathcal{S}_\mathrm{Ham}(\Delta)$ are intersections of images of $T$-orbit type strata (see, e.g., Example \ref{example:hamstrat}).
\end{remark}

\end{document}